\documentclass[11pt]{article}
\usepackage[a4paper,margin=28mm]{geometry}
\usepackage[T1]{fontenc}
\usepackage{lmodern,amsmath,amssymb,amsthm,mathtools,microtype}
\usepackage[colorlinks=true,linkcolor=blue,citecolor=blue,urlcolor=blue]{hyperref}
\newtheorem{theorem}{Theorem}
\newtheorem{lemma}[theorem]{Lemma}
\newtheorem{proposition}[theorem]{Proposition}
\newcommand{\R}{\mathbb R}

\newcommand{\I}{\mathbf I}
\newcommand{\M}{\mathbf M}

\newcommand{\asrk}{\operatorname{asrk}}
\newcommand{\Fill}{\operatorname{Fillvol}}
\newcommand{\sys}{\operatorname{sys}}
\newcommand{\inj}{\operatorname{inj}}
\newcommand{\wlim}{\hbox{$\omega$-$\lim$ }}

\title{CAT(0) spaces without linear filling\\at the asymptotic rank}
\author{Stephan Stadler}
\date{}
\begin{document}
\maketitle
\begin{abstract}
For every integer $\nu\ge2$, we construct a $(\nu+1)$-dimensional,
locally compact, geodesically complete CAT(0) space of asymptotic rank
$\nu$ which does not satisfy a linear isoperimetric inequality for
integral $\nu$-cycles. Its filling function is bounded below by
$c v\log v$ for all sufficiently large $v$.
\end{abstract}

\noindent{\bf 1. Introduction}
\medskip

\noindent Gromov's isoperimetric gap conjecture predicts a linear filling
inequality in dimensions at least the asymptotic rank for locally compact
cocompact CAT(0) spaces \cite[\S6.B$_2$, p.~128, (b)]{Gromov}.
More generally, one may ask whether every CAT(0) space of asymptotic
rank $\nu$ satisfies
\[
 \Fill_X(T)\le C\cdot\M(T)
\]
for integral $k$-cycles, $k\ge\max\{\nu,1\}$.
Here $\M(T)$ denotes the mass and $\Fill_X(T)$ the infimum of the
masses of its fillings. We write
\[
 \mathrm{FV}_{k+1}^X(v)=\sup\{\Fill_X(T):T\in\I_k(X),
 \partial T=0,\ \M(T)\le v\}.
\]
Wenger proved that the filling function in
these dimensions has sub-Euclidean growth
\cite[Theorem~1.2]{Wenger2011}.

In asymptotic rank at most two, Dru\c{t}u--Lang--Papasoglu--Stadler
\cite[Theorem~A and Corollary~B]{DLPS} proved filling inequalities
with exponent $1+\delta$, for every $\delta>0$, for Lipschitz
two-spheres in locally compact CAT(0) spaces and, with constants depending on
the genus, for closed surfaces. Lang--Stadler--Urech
\cite[Theorem~1.1]{LSU} obtained the corresponding inequality for
compactly supported integral $k$-cycles, $k\ge2$, assuming finite
asymptotic Nagata dimension. Recently, Peteranderl
\cite[Theorem~1.2]{Peteranderl} proved the linear inequality under
the same covering assumption, in every dimension at least the
asymptotic rank. The following example shows that this assumption
cannot be omitted.

\begin{theorem}\label{thm:main}
For every integer $\nu\ge2$, there exists a $(\nu+1)$-dimensional,
locally compact, geodesically complete CAT(0) space $X$ of asymptotic
rank $\nu$ and compactly supported integral $\nu$-cycles $T_n$ such that
\[
 \frac{\Fill_X(T_n)}{\M(T_n)}=\frac{n}{\nu+1}\longrightarrow\infty.
\]
Moreover, for some $c>0$ and all sufficiently large $v$,
\[
 \mathrm{FV}_{\nu+1}^X(v)\ge c v\log v.
\]
\end{theorem}

Thus a cocompact isometry group is essential for Gromov's conjecture.
\medskip

\noindent {\bf 2. Construction}\par\nopagebreak
\medskip

\noindent We use integral currents in the sense of Ambrosio--Kirchheim \cite{AK}.
We write $\I_k(X)$ for integral $k$-currents of finite mass and set
\[
 \Fill_X(T)=\inf\{\M(V):V\in\I_{k+1}(X),\ \partial V=T\}.
\]

The space $X$ will be constructed by attaching a sequence of building blocks $Y_n$ to a line $L$.
Each building block is glued from a truncated Euclidean cone over a closed manifold and a
hyperbolic funnel. Let $Z$ be a compact CAT(1) length space. Its
Euclidean and hyperbolic cones
\[
 E=C_0(Z)=[0,\infty)\times_r Z,\qquad
 H=C_{-1}(Z)=[0,\infty)\times_{\sinh s}Z
\]
are CAT(0) and CAT$(-1)$, respectively
\cite[Theorem~II.3.14]{BH}. Fix $R>0$ and put
$s_0=\operatorname{arsinh}R$. Glue $\{r\le R\}\subset E$ to
$\{s\ge s_0\}\subset H$, identifying boundary points with the same
$Z$-coordinate, and take the length metric. Both boundary length
metrics are $R\cdot d_Z$. Denote the resulting space by $Y(Z,R)$,
its vertex by $o$, and its Euclidean core by $K= \{r\le R\}$. Its warping function,
in the radial coordinate from $o$, is
\[
 F_R(r)=
 \begin{cases}
 r,&0\le r\le R,\\
 \sinh(r-R+s_0),&r\ge R.
 \end{cases}
\]
This function is increasing and convex, with
$F_R'(0+)=F_R'(R-)=1$ and $F_R'(R+)=\sqrt{1+R^2}$.
Alexander--Bishop's warped-product criterion
\cite[Theorem~2.2]{AB} therefore shows that $Y(Z,R)$ is CAT(0).
It is locally compact, and $K=\overline B_R(o)$ is a convex ball with its
original Euclidean cone metric. Outside $K$, the metric is locally
CAT$(-1)$.  Hence every asymptotic cone
of $Y(Z,R)$ is a  tree.

Fix an integer $\nu\ge2$. Choose closed oriented hyperbolic surfaces
$(\Sigma_n,g_n)$ of genus $\gamma_n$ such that
\[
 \sys(\Sigma_n)>2\pi n,\qquad \gamma_n\asymp e^{3\pi n/2}.
\]
Indeed, \cite[Theorem~1.5]{KSV} gives
$\sys(\Sigma)\ge\frac43\log\gamma-C_0$ along a fixed prime-power
congruence tower. Successive covering degrees in this tower are
uniformly bounded; see \cite[\S4]{KSV}. Taking the first level with
genus at least $\exp(\frac34(2\pi n+C_0+1))$ gives the required surfaces.
Choose flat tori $Q_n=(\R/(4\pi n\mathbb Z))^{\nu-2}$.
Give $N_n=\Sigma_n\times Q_n$
the product metric $\bar g_n$, with $Q_n$ a point when $\nu=2$, and put
\[
 Z_n=(N_n,n^{-2}\bar g_n),\qquad
 Y_n=Y(Z_n,n),\qquad K_n=\{r\le n\}\subset Y_n.
\]
Since the injectivity radius satisfies $\inj(N_n)>\pi n$ and $N_n$
is locally CAT(0), the space $Z_n$ is CAT(1).
Attach the vertex $o_n$ of $Y_n$ to the integer $n$ of a complete
line $L=\R$, and equip the union $X$ with its path metric. Define the {\em core} of $X$ by
\[
 C=L\cup\bigcup_{n\ge1}K_n.
\]
Reshetnyak's gluing theorem \cite[Theorem~II.11.1]{BH} shows that
$X$ is CAT(0). Local finiteness of the attachments gives local compactness.
Away from its vertex, $Y_n$ is a CAT(0) $(\nu+1)$-manifold and hence
locally geodesically complete. At $o_n$, every direction has an
antipode since $\inj(Z_n)>\pi$. Thus $Y_n$, and consequently $X$,
is geodesically complete.
The subsets $C$ and $K_n$ are closed and convex.
For each fixed $n$, every asymptotic cone of $Y_n$ is a  tree.

Since $X$ has Hausdorff dimension $\nu+1$, and
every integral cycle in $X$ bounds \cite[Corollary~1.4]{Wenger2005},
integral $(\nu+1)$-cycles vanish and integral $\nu$-cycles have unique
fillings. Let $A_n$ be the integral current induced by the oriented
core $K_n$, and put $T_n=\partial A_n$. Uniqueness and the
Euclidean cone metric give
\[
 \frac{\Fill_X(T_n)}{\M(T_n)}
 =\frac{\M(A_n)}{\M(T_n)}=\frac{n}{\nu+1}.
\]
Moreover, Gauss--Bonnet gives
\[
 a_n:=\M(T_n)=4\pi(\gamma_n-1)(4\pi n)^{\nu-2}
 \asymp n^{\nu-2}e^{3\pi n/2}.
\]
Thus $n\asymp\log a_n$ and
$\mathrm{FV}_{\nu+1}^X(a_n)\ge c a_n\log a_n$ for some $c>0$.
Since $a_n$ is nondecreasing and tends to infinity with bounded
successive ratios, monotonicity gives the asserted lower bound for
all sufficiently large $v$.

It remains to show that $X$ has asymptotic rank $\nu$. We first consider
limits based far from both the line $L$ and the hyperbolic funnels.


\begin{lemma}\label{lem:core-limits}
Let $\lambda_k\to0$ and $x_k=(r_k,u_k)\in K_{n_k}$ satisfy
\begin{equation}\label{eq:deep-core}
 \wlim\lambda_kr_k=\infty,\qquad
 \wlim\lambda_k(n_k-r_k)=\infty.
\end{equation}
Then $\wlim(\lambda_kX,x_k)$ is isometric to $\R^{\nu-1}\times T$
for a tree $T$ containing a line. In particular, it has dimension $\nu$.
\end{lemma}
\begin{proof}
Every ball of fixed rescaled radius about $x_k$ eventually lies in
$K_{n_k}$. Put $h_k=\lambda_kr_k$ and
$b_k=\lambda_kr_k/n_k\le\lambda_k$. The rescaled intrinsic cross
section through $x_k$ is $(N_{n_k},b_k^2\bar g_{n_k})$.
Its surface factor has curvature $-b_k^{-2}\to-\infty$ and injectivity
radius greater than $\pi h_k$, with $\wlim h_k=\infty$. Its pointed
ultralimit is therefore a tree $T$ containing a line
\cite[Proposition~III.H.1.2]{BH}. The torus factor has circle lengths
$4\pi h_k$ and hence converges to $\R^{\nu-2}$. On bounded rescaled
radial intervals the relative warping factor is $1+t/h_k$, with
$\wlim(1+t/h_k)=1$. Since both radial endpoints escape, the asymptotic
cone splits as $T\times\R^{\nu-1}$.
\end{proof}

\begin{proposition}\label{prop:rank}
The space $X$ has asymptotic rank $\nu$.
\end{proposition}
\begin{proof}
By Wenger's characterization \cite[Proposition~3.1]{Wenger2011},
if $\asrk X>\nu$, some asymptotic cone contains a Euclidean $(\nu+1)$-ball
$B$. A closed ball disjoint from $C$ lies in a hyperbolic funnel
and is CAT$(-1)$. Thus points at
positive limiting distance from $C$ have tree neighborhoods.
It follows that $B$ lies in the core ultralimit $C_\omega$.

The line $L$ limits either to a line $L_\omega$ or to the empty set. 
Thus after decreasing the ball and rescaling, we may assume that $B$ has radius equal to one and has positive distance to $L_\omega$. Hence $B$ can be represented  as an ultralimit of compact sets $B=\wlim B_k$ where
each $B_k$ is contained in a single core $K_{n_k}$. Apply the radial
homotheties $(r,u)\mapsto(r/2,u)$, pass to the corresponding pointed
ultralimit, and renormalize. Since $n_k-r_k/2\ge r_k/2$, we may also
assume that the center has positive rescaled distance from the outer
boundaries. Denote by $x_\omega$ the center of $B$. Then,
the rescaled distances from the representatives $x_k\in B_k$ of $x_\omega$ to both, the vertices $o_{n_k}$ and the boundaries of $K_{n_k}$ are bounded below by a constant  $c>0$. Changing the scaling factors of the asymptotic cone we obtain for every $l\in\mathbb N$ a new asymptotic cone $W_l$ of $X$ which contains a Euclidean $(\nu+1)$-ball of radius $l$ that is a limit of compact sets at distance at least $lc$ from the corresponding vertices and boundaries. By a diagonal argument, we now obtain an asymptotic cone $(X_\omega,x_\omega)=\wlim(\lambda_k\cdot X,x_k)$ with a Euclidean $(\nu+1)$-ball and with  $x_k=(r_k,u_k)\in K_{n_k}$, satisfying $\wlim\lambda_kr_k=\infty$ and
$\wlim\lambda_k(n_k-r_k)=\infty$.  This contradicts Lemma~\ref{lem:core-limits} and shows $\asrk X\leq \nu$.
Conversely, choose $x_n=(n/2,u_n)\in K_n$ and
$\lambda_n=n^{-1/2}$. Both conditions in \eqref{eq:deep-core} hold,
so Lemma~\ref{lem:core-limits} gives an asymptotic cone
$\R^{\nu-1}\times T$ containing a Euclidean $\nu$-flat.
Hence $\asrk X\ge\nu$.
\end{proof}

The theorem follows. By \cite[Theorem~1.2]{Peteranderl}, the space
$X$ has infinite asymptotic Nagata dimension. One can show that even
the asymptotic dimension of $X$ is infinite, cf. \cite{LS}.
For $\nu=2$, the boundaries of the cores are closed surfaces of
unbounded genus. Their filling-to-area ratios tend to infinity, so
a linear version of \cite[Corollary~B]{DLPS} cannot hold with a constant
independent of the genus.
For $\nu=2$, the example also shows that J{\o}rgensen--Lang's theorem
\cite[Theorem~3]{JL}, asserting that every $3$-dimensional Hadamard
manifold has Nagata dimension $3$, and its extension to 
CAT(0) $3$-manifolds by Papasoglu--Swenson
\cite[Theorem~4.4]{PS}, do not extend to locally compact,
geodesically complete $3$-dimensional CAT(0) spaces.

\smallskip
\noindent\textit{Acknowledgements.}
I thank Urs Lang, Panos Papasoglu, Jonas Peteranderl and Stefan Wenger for comments and feedback.

\smallskip
\noindent\textit{Declaration on AI use.}
ChatGPT assisted with drafting and editing this manuscript. The author takes full responsibility for the mathematical content and
the final text.

\end{document}